\documentclass[10pt, article]{amsart}

\usepackage{tikz}
\usetikzlibrary{calc}
\usepackage{ae} 
\usepackage[T1]{fontenc}
\usepackage[cp1250]{inputenc}
\usepackage{amsmath}
\usepackage{amssymb, amsfonts,amscd,verbatim}

\usepackage[normalem]{ulem}
\usepackage{hyperref}
\usepackage{indentfirst}
\usepackage{latexsym}
\input xy
\xyoption{all}

\usepackage{amsmath}    

\theoremstyle{plain}
\newtheorem{Pocz}{Poczatek}[section]
\newtheorem{Proposition}[Pocz]{Proposition}
\newtheorem{Theorem}[Pocz]{Theorem}
\newtheorem{Corollary}[Pocz]{Corollary}

\newtheorem{Lemma}[Pocz]{Lemma}

\newtheorem{Problem}[Pocz]{Problem}
\newtheorem{Example}[Pocz]{Example}

\theoremstyle{definition}
\newtheorem{Definition}[Pocz]{Definition}

\theoremstyle{remark}
\newtheorem{Remark}[Pocz]{Remark}

\numberwithin{equation}{section}
\title[Overlays and group actions]
{Overlays and group actions}

\author{Jerzy Dydak}

\address{University of Tennessee, Knoxville, TN 37996}
\email{dydak\@@math.utk.edu}

\date{ \today
}
\keywords{covering maps, overlays, paracompactness}

\subjclass[2000]{Primary 54F45; Secondary 55M10}

\begin{document}
\maketitle
\begin{center}
\today
\end{center}

\begin{abstract}

Overlays were introduced by R.H.Fox \cite{Fox1}  as a subclass of covering maps. We offer a different view of overlays: it resembles the definition of paracompact spaces via star refinements of open covers. One introduces covering structures for covering maps and $p:X\to Y$ is an overlay if it has a covering structure that has a star refinement.

We prove two characterizations of overlays: one using existence and uniqueness of lifts of discrete chains, the second as maps inducing simplicial coverings of nerves of certain covers. We use those characterizations to improve results of Eda-Matijevi\' c concerning topological group structures on domains of overlays whose range is a compact topological group.

In case of surjective maps $p:X\to Y$ between connected metrizable spaces we characterize overlays as local isometries:  $p:X\to Y$ is an overlay if and only if one can metrize $X$ and $Y$ in such a way that $p|B(x,1):B(x,1)\to B(p(x),1)$ is an isometry for each $x\in X$.

\end{abstract}

\section{Introduction}

R.H.Fox \cite{Fox1} introduced overlays $p:X\to Y$ as maps such that $Y$ has an open cover $\mathcal{U} $ with the property
that each $U\in \mathcal{U} $ is evenly covered and there is a set $S$ so that each $p^{-1}(U)$, $U\in \mathcal{U} $, 
can be decomposed as a disjoint union $\bigcup_{s\in S}U_s$ with $p|U_s:U_s\to U$ being a homeomorphism. Moreover, if $U,V\in \mathcal{U} $ intersect, then there is a reindexing of elements of decompositions of preimages $p^{-1}(U)$, $p^{-1}(V)$ so that $U_s\cap V_t\ne\emptyset$ implies $s=t$.

\begin{Remark}
The definitions of overlays in \cite{Moo}  (Definition 1.1) and \cite{Mel}  (the text prior to Proposition 7.2) must be read in the spirit of the above definition. Namely, they do not mean that each $p^{-1}(U)$, $U\in \mathcal{U} $,
has a fixed decomposition as a disjoint union $\bigcup_{s\in S}U_s$ with $p|U_s:U_s\to U$ being a homeomorphism. The decomposition is fixed in terms of sets, not in terms of indexing by elements of $S$.
\end{Remark}

The reason overlays are needed is that for general topological spaces one cannot build a theory of covering maps similarly to that for locally connected spaces (see examples in \cite{HilWyl}  and \cite{Dyd}). 

In this paper we offer a different view of overlays: it resembles the definition of paracompact spaces via star refinements of open covers. One introduces covering structures for covering maps similarly to overlay structures as in 
\cite{Moo}. $p:X\to Y$ is an overlay if it has a covering structure that has a star refinement.

We prove two characterizations of overlays:

1. $p:X\to Y$ is an overlay if and only if there is an open cover $\mathcal{S} $ of $X$ such that every
$\mathcal{U} $-chain, $\mathcal{U} =p(\mathcal{S} )$, has a lift that is an $\mathcal{S} $-chain and that lift is unique.

2. $p:X\to Y$ is an overlay if and only if there is an open cover $\mathcal{S} $ of $X$ such that for
 $\mathcal{U} =p(\mathcal{S} )$ the induced map $\mathcal{N} (\mathcal{S} )\to \mathcal{N} (\mathcal{U} )$ of nerves of covers is a simplicial covering.

Characterization 1 uses ideas of Berestovskii-Plaut
\cite{BP3} later expanded in
\cite{BDLM}.

In case of surjective maps $p:X\to Y$ between connected metrizable spaces we characterize overlays as local isometries:  $p:X\to Y$ is an overlay if and only if one can metrize $X$ and $Y$ in such a way that $p|B(x,1):B(x,1)\to B(p(x),1)$ is an isometry for each $x\in X$.

The author is grateful to Katsuya Eda and Vlasta Matijevi \'c for several comments that improved the exposition of the paper.

\section{Covering actions}
\begin{Definition}
Given a free action of a group $G$ on a topological space $X$, by a \textbf{slice} we mean an open subset $U$ of $X$ such that
$U\cap (g\cdot U)\ne\emptyset$ implies $g=1_G$.
\end{Definition}

\begin{Definition}
A \textbf{covering action} of a group $G$ on a topological space $X$ is a free action
such that $X$ can be covered by slices.
\end{Definition}

Notice every covering action of $G$ on $X$ induces a covering map $p:X\to X/G$ from $X$ to the space of orbits $X/G$ provided it is given the quotient topology. 

\begin{Proposition}
If $G$ acts freely on a topological space $X$ and the induced map $p:X\to X/G$ is a covering map, then the action is a covering action.
\end{Proposition}
\begin{proof}
Suppose $V\subset X/G$ is open and evenly covered by $p$. Pick $U\subset p^{-1}(V)$ such that $p|U:U\to V$ is a homeomorphism. If $x\in U\cap (g\cdot U)$, then $U$ contains both $x$ and $g^{-1}\cdot x$. As $p|U$ is injective, $x=g^{-1}\cdot x$ resulting in $g=1_G$. Thus $U$ is a slice of the action and the action is a covering action.
\end{proof}

We extend the concept of a slice to covering maps:
\begin{Definition}
Given a covering map $p:X\to Y$, a \textbf{slice} $U$ of $p$ is an open subset of $X$ such that
$p^{-1}(p(U))$ can be decomposed into a disjoint union $\{U_s\}_{s\in S}$ of open subsets of $X$ with the following properties:\\
1. $p|U_s:U_s\to p(U_s)$ is a homeomorphism of $U_s$ onto $V=p(U)$ for each $s\in S$,\\
2. $U_t=U$ for some $t\in S$.\\
$U$ is also going to be called a \textbf{slice over} $p(U)$.
\end{Definition}

That leads to the following

\begin{Definition}\label{coveringstructureDef}
Given a covering map $p:X\to Y$, a \textbf{covering structure} $\mathcal{S} $ of $p$ is an open cover of $X$ by slices of $p$ such that for each $U\in \mathcal{S} $,
$p^{-1}(p(U))$ can be decomposed into a disjoint union $\{U_j\}_{j\in J}$ of elements of $\mathcal{S} $ satisfying $p(U_j)=p(U)$ for each $j\in J$.
\end{Definition}

\section{Overlay actions}

\begin{Definition}
An \textbf{overlay action} of a group $G$ on a topological space $X$ is a free action such that $X$ has a
covering structure $\mathcal{U}$ with the property that for any $x\in X$ the open star $st(x,\mathcal{U})=\bigcup\{U\in \mathcal{U}  | x\in U\}$ is a slice of the action.\\
In that case $\mathcal{U}$ is called an \textbf{overlay structure} of the action.
\end{Definition}

\begin{Problem}
Is every covering action an overlay action? What about finite group actions?
\end{Problem}

\begin{Proposition}
Given a free action of a group $G$ on a topological space $X$ the following conditions are equivalent:\\
a. the action is an overlay action,\\
b. $X$ has an open
cover $\mathcal{U}$ with the property that for any pair $U,V\in\mathcal{U}$ there is $g\in G$ making $U\cup g\cdot V$ a slice of the action.
\end{Proposition}
\begin{proof}
\textbf{a)$\implies$b).} Choose  a
cover structure $\mathcal{U}$ with the property that for any $x\in X$ the open star $st(x,\mathcal{U})$ a slice of the action.
Given two elements $U,V\in \mathcal{U}$ either there is no $g\in G$ such that $U\cap g\cdot V\ne\emptyset$, in which case $U\cup V$ is a slice of the action, or there is $g\in G$ such that $x\in U\cap g\cdot V$ for some $x\in X$.
In this case $g\cdot V\in \mathcal{U} $ and $U\cup g\cdot V\subset st(x,\mathcal{U} )$, so $U\cup g\cdot V$ is a slice of the action.\par
\textbf{b)$\implies$a).} Let $\mathcal{V} $ be the cover of $X$ consisting of $g\cdot U$, $g\in G$ and $U\in \mathcal{U} $. $\mathcal{V} $ is an overlay structure on $X$. Indeed, given $x\in X$ and given $y,z\in st(x,\mathcal{V} )$
such that $z=h\cdot y$ for some $h\in G$, there exist $U_1, U_2\in \mathcal{V} $ so that $x\in g_1\cdot U_1\cap g_2\cdot U_2$ and
$y\in g_1\cdot U_1$, $z\in g_2\cdot U_2$. Choose $g\in G$ so that $U_1\cup g\cdot U_2$ is a slice of the action. Since this set contains both $g_1^{-1}\cdot x$ and $g\cdot g_2^{-1}\cdot x$, $g_1^{-1}=g\cdot g_2^{-1}$.
Now, $g_1^{-1}\cdot y\in U_1$, $g\cdot g_2^{-1}\cdot z\in g\cdot U_2$, so they must be equal resulting in $y=z$. That means $st(x,\mathcal{V} )$ is a slice of the action. Clearly, $g\cdot st(x,\mathcal{V} )=st(g\cdot x,\mathcal{V} )$, so the family $st(x,\mathcal{V} )$, $x\in X$, is a cover structure on $X$.
\end{proof}

\begin{Proposition}\label{BasicOverlayAction}
Suppose $G$ acts on a topological space $X$. If $X/G$ has an open cover $\mathcal{U}$ consisting of connected subsets such that for each pair $U,V\in\mathcal{U}$ the union $U\cup V$ is evenly covered by the projection $p:X\to X/G$,
then the action is an overlay action.
\end{Proposition}
\begin{proof}
Notice slices over each element of $\mathcal{U}$ are uniquely determined leading to a cover $\mathcal{V}$ of $X$
by slices of the action.
Choose $U,V\in\mathcal{V}$. If $p(U)\cap p(V)=\emptyset$, then clearly $U\cup V$ is a slice of the action.\\
Suppose $y\in p(U)\cap p(V)$, pick $x\in U\cap p^{-1}(y)$ and choose $g\in G$ so that $g^{-1}\cdot x\in V$.
We claim $U\cup g\cdot V$ is a slice of the action. Indeed, $p(U)\cup p(V)$ is evenly covered and the slice over it
containing $x$ has to contain both $U$ and $g\cdot V$. Therefore $U\cup g\cdot V$ is a slice of the action.
\end{proof}

\begin{Proposition}\label{OverlayStructureForHomomorphisms}
If $H$ is a normal, closed, and discrete subgroup of a topological group $G$, then the natural right (respectively, left) action of $H$ on $G$ has an overlay structure in the form of $\{U\cdot g\}_{g\in G}$ for any open, symmetric neighborhood $U$ of $1_G$ satisfying $U^4\cap H=\{1_G\}$.
\end{Proposition}
\begin{proof}

\textbf{Claim 1}: If $(U\cdot g)\cap (U\cdot f)\ne \emptyset$, then $g\cdot f^{-1}\in U\cdot U$.\\
\textbf{Proof of Claim 1}: $u\cdot g=v\cdot f$ results in $g\cdot f^{-1}=u^{-1}\cdot v\in U^{-1}\cdot U=U\cdot U$. $\Box$.\\
\textbf{Claim 2}: If $(U\cdot g)\cap (U\cdot f)\ne \emptyset$ and $(U\cdot g)\cap (U\cdot f\cdot h)\ne \emptyset$
for some $h\in H$, then $h=1_G$.\\
\textbf{Proof of Claim 2}:
By Claim 1,  $g\cdot f^{-1}\in U\cdot U$ and $g\cdot (f\cdot h)^{-1}\in U\cdot U$.
Consequently, $g\cdot (f\cdot h)^{-1}=g\cdot f^{-1}\cdot (f\cdot h^{-1}\cdot f^{-1})\in U\cdot U$
resulting in $f\cdot h^{-1}\cdot f^{-1}\in U^4\cap H$. Hence $h=1_H$. $\Box$.\\
\textbf{Proof of \ref{OverlayStructureForHomomorphisms}:} Given $U\cdot g$ and $U\cdot f$ their union is clearly
a slice of the action if $p(U\cdot g)\cap p(U\cdot f)=\emptyset$. Otherwise we may consider the case
$(U\cdot g)\cap (U\cdot f)\ne\emptyset$ and their union is not a slice of the action if there is $h\in H\setminus\{1_G\}$ such that $(U\cdot g)\cap (U\cdot f\cdot h)\ne\emptyset$. That contradicts Claim 2.
\end{proof}

\begin{Problem}
If $H$ is a closed discrete subgroup of a topological group $G$, then
is the right action of $H$ on $G$ an overlay action?
\end{Problem}

\begin{Definition}
If an action of $G$ is an overlay action, then the image of any overlay structure is called an \textbf{overlay cover} of $X/G$.
\end{Definition}

\begin{Corollary}\label{OverlayCoversForHomomorphisms}
If $p:X\to Y$ is a continuous epimorphism of topological groups that is a homeomorphism when restricted to some non-empty open set $V\subset X$, then there is an open neighborhood $W$ of $1_Y$ such that the family $\{W\cdot y\}_{y\in Y}$ is an overlay cover of $Y$.
\end{Corollary}
\begin{proof}
Let $H=ker(p)$. Notice $H$ is normal, closed, and discrete subgroup of $X$. By \ref{OverlayStructureForHomomorphisms} there is an open neighborhood $U$ of $1_X$ in $X$ such that
the family $\{U\cdot x\}_{x\in X}$ is an overlay structure of $p$. Put $W=p(U)$ and notice the image of the family $\{U\cdot x\}_{x\in X}$ is $\{W\cdot y\}_{y\in Y}$.
\end{proof}

\section{Overlays}
\begin{Definition}
Given a map $p:X\to Y$, an \textbf{overlay structure} of $p$ is a covering structure $\mathcal{S}$ of $X$ 
such that
for any $x\in X$ the star $st(x,\mathcal{S})$ is a slice of $p$.\\
The image of $\mathcal{S}$ is called an \textbf{overlay cover} of $Y$.\\
$p$ is an \textbf{overlay} if it has an overlay structure.
\end{Definition}

\begin{Example} Consider the covering map $p:R\to S^1$ given by $p(t)=\exp(2\pi t i)$. Notice that the family of open intervals $\{(a,a+1)\}_{a\in R}$ is a covering structure of $p$ but is not an overlay structure of $p$.
\end{Example}

\begin{Corollary}
An action of $G$ on a topological space $X$ is an overlay action if and only if $p:X\to X/G$ is an overlay.
\end{Corollary}

\begin{Proposition}\label{CharOfOverlayStructures}
Suppose $\mathcal{S}$ is a covering structure of $p:X\to Y$. The following conditions are equivalent:\\
a. $\mathcal{S}$ is an overlay structure,\\
b. If a slice $U\in \mathcal{S}$ intersects two different slices $V,W\in \mathcal{S}$, then $p(V)\ne p(W)$,\\
c. If two slices $U,V\in \mathcal{S}$ intersect, then $U\cap V$ is a slice over $p(U)\cap p(V)$.
\end{Proposition}
\begin{proof}
\textbf{a)$\implies$b)}. If a slice $U$ in $\mathcal{S}$ intersects two different slices $V,W\in \mathcal{S}$ such that $p(V)=p(W)$,
then we pick $x\in U\cap V$ and notice it causes a contradiction as $V\cap p^{-1}(p(U\cap W))$ and $U\cap W$
are disjoint, non-empty, contained in $U\cup V$, and with equal images under $p$.\\
\textbf{b)$\implies$c)}. Suppose two slices $U,V\in \mathcal{S}$ intersect and $x\in U$ with $p(x)\in p(U)\cap p(V)$.
We claim $x\in V$ which is sufficient to conclude that $U\cap V$ is a slice over $p(U)\cap p(V)$. Indeed, if $x\notin V$,
then there is a slice $W\in \mathcal{S}$ over $p(V)$ containing $x$. That contradicts
$x\in U\cap W$ and $U\cap V\ne\emptyset$.\\
\textbf{c)$\implies$a)}. Pick $x\in X$. It suffices to show $p|st(x,\mathcal{S} )$ is injective. Suppose
$y,z\in st(x,\mathcal{S} )$ and $p(y)=p(z)$.
Choose $U,V\in \mathcal{S}$ containing $x$ such that $y\in U$ and $z\in V$. Since $p(y)\in p(U)\cap p(V)$, there is $t\in U\cap V$ satisfying $p(t)=p(y)$. Hence $t=y$ and $t=z$ resulting in $y=z$.
\end{proof}

\begin{Corollary}
Our definition of overlays coincides with that of Fox provided fibers of $p:X\to Y$ have the same cardinality.
\end{Corollary}
\begin{proof}
Fox's definition amounts to Condition b) in \ref{CharOfOverlayStructures} plus assumption of fibers of $p$ having the same cardinality.
\end{proof}

\begin{Remark}
See \cite{Fox2}  or \cite{Moo}  for examples of covering maps that are not overlays.
\end{Remark}

Here is a basic example of an overlay structure:
\begin{Example}\label{BasicExampleOfOverlays}
Suppose $p:X\to Y$ is a covering map.
If $\mathcal{U}$ is an open cover of $Y$ consisting of connected sets such that $U\cup V$ is evenly covered for any $U,V\in \mathcal{U}$, then $\mathcal{U}$ is an overlay cover and $\mathcal{S}$ consisting of components of $p^{-1}(U)$, $U\in\mathcal{U}$, is an overaly structure.
\end{Example}
\begin{proof}
Same as in \ref{BasicOverlayAction}.
\end{proof}

\begin{Corollary}[Marde\v si\' c-Matijevi\'c \cite{MV} , Fox \cite{Fox2} in the metric case]
If $p:X\to Y$ is a covering map and $Y$ is locally connected paracompact, then $p$ is an overlay.
\end{Corollary}
\begin{proof}
Consider an open covering of $Y$ consisting of connected subsets that are evenly covered by $p$. Choose a star refinement $\mathcal{U}$ of that cover consisting of connected subsets of $Y$. By \ref{BasicExampleOfOverlays} it is an overlay cover.
\end{proof}

It is clear that a subset of an evenly covered set is evenly covered. The corresponding result for overlay covers is less obvious.

\begin{Lemma}
Any open refinement of an overlay cover is an overlay cover.
\end{Lemma}
\begin{proof} 
Given $W\subset U$ and given a decomposition of $p^{-1}(U)$ into slices $\{U_s\}_{s\in S}$ over $U$, we
define $W_s:=U_s\cap p^{-1}(W)$. If $W\subset V$, where $V$ is another member of the overlay cover of $Y$,
we know the slices over $V$ are of the form $V_s$, $s\in S$, such that $V_s$ intersects only $U_s$ among slices over $U$. Therefore $U_s\cap p^{-1}(W)=V_s\cap p^{-1}(W)$ for each $s\in S$ and the choice of slices over $W$ is unique.

This recipe allows for creation of an overlay structure over $\mathcal{V}$ if it is a refinement of an overlay cover $\mathcal{U}$ with a given overlay structure over it.
\end{proof}

\begin{Lemma}\label{FiniteCoveringsLemma}
If fibers of a covering map $p:X\to Y$ are finite and $X$ is Hausdorff, then any open cover $\mathcal{V} $ of $X$ can be refined by a covering structure.
\end{Lemma}
\begin{proof}
We may assume that each element of $\mathcal{V} $ is a slice of $p$. For each $y\in Y$ choose $V(x)\in \mathcal{V} $ for all $x\in p^{-1}(y)$, then put $U(y)=\bigcap\limits_{x\in p^{-1}(y)} p(V(x))$. Put $W(x)=V(x)\cap p^{-1}(U(p(x))$ for $x\in X$. If $p(x)=p(x')$ and $x\ne x'$, then $W(x)\cap W(x')$ is open-closed in both $W(x)$ and $W(x')$.
Remove those sets forming $$S(x)=W(x)\setminus \bigcup\limits_{x'\in p^{-1}(p(x))\setminus \{x\}} W(x)\cap W(x')$$ and
notice $\{S(x))\}_{x\in X}$ is a covering structure of $p$ refining $\mathcal{V} $.
\end{proof}

\begin{Corollary}[V.Matijevi\' c \cite{VMat}]
If $p:X\to Y$ is a covering projection with finite fibers and $X$ is paracompact, then $p$ is an overlay.
\end{Corollary}
\begin{proof}
Given a covering structure of $p$ pick its star refinement and then refine it by a covering structure $\mathcal{S} $ using \ref{FiniteCoveringsLemma}. Clearly, $\mathcal{S} $ is an overlay structure of $p$. 
\end{proof}

Here is another basic example of overlays in case of metric spaces.
\begin{Example}\label{LocalIsometryExample}
If $p:X\to Y$ is a surjective map such that for some $r > 0$ the induced map $p:B(x,r)\to B(p(x),r)$ is an isometry for each $x\in X$, then $p$ is an overlay.
\end{Example}
\begin{proof}
Consider $\mathcal{S} =\{B(x,r/3)| x\in X\}$. Notice it is a covering structure. As $st(x,\mathcal{S} )\subset B(x,r)$ for each $x\in X$, it is an overlay structure.
\end{proof}

\begin{Remark}
Lemma 1.4 of \cite{Moo} shows it is sufficient to assume $r$ depends on $x\in X$. However, our interest is in the converse of  \ref{LocalIsometryExample} - see \ref{LocalIsometryTheorem}.
\end{Remark}

\section{Chain lifting}
\begin{Definition}
Given a cover $\mathcal{U}$ of $Y$, by a $\mathcal{U}$-\textbf{chain} $\{y_0,\ldots,y_n\}$ we mean
a finite sequence of points of $Y$ such that for each $0\leq i < n$ there is $U\in \mathcal{U}$ containing both $y_i$ and $y_{i+1}$.\\
A $\mathcal{U}$-chain $\{y_0,\ldots,y_n\}$ is called a $\mathcal{U}$-\textbf{loop} if $y_0=y_n$.
\end{Definition}

\begin{Definition}
Suppose $p:X\to Y$ is a surjective function and $\{y_0,\ldots,y_n\}$ is a chain of points in $Y$.
A chain $\{x_0,\ldots,x_n\}$ in $X$ is a \textbf{lift} of $\{y_0,\ldots,y_n\}$ if $p(x_i)=y_i$ for $0\leq i\leq n$.\\
\end{Definition}

\begin{Theorem}
Suppose $p:X\to Y$ is an open surjective map. If $\mathcal{S}$ is an open cover of $X$ and $\mathcal{U}=p(\mathcal{S})$, then the following conditions are equivalent:\\
a. $\mathcal{S}$ is an overlay structure of $p:X\to Y$,\\
b. Given a $\mathcal{U}$-chain $\{y_0,\ldots,y_n\}$, and given
$x_0\in X$ such that $p(x_0)=y_0$, there is a unique lift $\{x_0,\ldots,x_n\}$ of $\{y_0,\ldots,y_n\}$
that is an $\mathcal{S}$-chain.
\end{Theorem}
\begin{proof}
\textbf{a)$\implies$b).} It suffices to consider $n=1$ as the general case follows by induction.
If $\{y_0,y_1\}$ is a $\mathcal{U}$-chain and $p(x_0)=y_0$, then we pick $U\in\mathcal{S}$ with $y_0,y_1\in p(U)$.
Next, we pick a slice $V$ over $p(U)$ containing $x_0$. Put $x_1=p^{-1}(y_1)\cap V$. That shows existence of a lift of $\{y_0,y_1\}$ that is an $\mathcal{S}$-chain. Suppose another $\mathcal{S}$-chain $\{x_0,x\}$ is a lift of $\{y_0,y_1\}$. Choose $W\in \mathcal{S}$ containing $x_0$ and $x$. By \ref{CharOfOverlayStructures}, $V\cap W$ is a slice over $p(V)\cap p(W)$. As $y_1\in p(V)\cap p(W)$, $x, x_1\in V\cap W$ which implies $x=x_1$. That shows uniqueness of $\mathcal{S}$-lifts.

\textbf{b)$\implies$a).} Notice $p|U$ is injective for each $U\in\mathcal{S}$ as otherwise we have non-uniqueness of lifts of chains of the type $\{y,y\}$. 

Suppose $U, V\in\mathcal{S}$ contain $x_0$ and $y_1\in p(U)\cap p(V)$.
Let $\{x_0,x_1\}$ be an $\mathcal{S}$-lift of $\{p(x_0),y_1\}$. If $x_1\notin U$, there is another $\mathcal{S}$-lift of $\{p(x_0),y_1\}$, a contradiction. Thus $x_1\in U\cap V$ proving $p:U\cap V\to p(U)\cap p(V)$ is surjective, hence a homeomorphism.
\end{proof}

Here is a converse to \ref{LocalIsometryExample}:
\begin{Theorem}\label{LocalIsometryTheorem}
If $p:X\to Y$ is an overlay of connected metrizable spaces, then one can metrize $X$ and $Y$ in such a way that $p|B(x,1):B(x,1)\to B(p(x),1)$ is an isometry for each $x\in X$.
\end{Theorem}
\begin{proof}
Pick an overlay cover $\mathcal{U}$. Choose a partition of unity $\{\phi_j\}_{\in J}$ on $Y$ subordinate to a star-refinement $\mathcal{V} $ of $\mathcal{U} $. Given a metric $d$ on $Y$ define a new metric by the formula $d_Y(y,z)=d(y,z)+\sum_{j\in J} |\phi_j(y)-\phi_j(z)|$ and notice it is equivalent to $d$.
If $y$ and $z$ do not belong to the same element of $\mathcal{V} $, then either $\phi_j(y)=0$ or $\phi_j(z)=0$ for each $j\in J$. Therefore $d_Y(y,z) > 2$. Consequently, each ball $B(y,2)$ has to be contained in an element of $\mathcal{U} $ as it is contained in $st(y,\mathcal{V} )$. 

Pick an overlay structure $\mathcal{S} $ over $\{B(y,2)\}_{y\in Y}$.
Define a metric $d_X$ on $X$ as follows: $d_X(x,x')$ is the infimum of $\sum_{i=0}^n d_Y(p(x_i),p(x_{i+1}))$ over all $\mathcal{S} $-chains $\{x_0,\ldots,x_{n+1}\}$ joining $x$ and $x'$. 

Suppose $p(x)=p(x')$, $x\ne x'$, and $d_X(x,x') < 2$. Pick an $\mathcal{S} $-chain $\{x_0=x,x_1,\ldots, x_{n+1}=x'\}$ such that 
$\sum_{i=0}^n d_Y(p(x_i),p(x_{i+1})) < 2$. That means all $p(x_i)\in B(p(x),2)$, so $\mathcal{S} $-lifts of  $\{p(x_0),p(x_1),\ldots, p(x_{n+1})\}$ must be $\mathcal{S} $-loops, a contradiction. That shows $p|B(x,1):B(x,1)\to B(p(x),1)$ is an isometry for each $x\in X$.

Notice $d_X$ induces the same topology on $X$ as the original topology.
\end{proof}

\begin{Proposition}\label{DeckTransformationFactorization}
Given an overlay structure $\mathcal{S}$ of $p:X\to Y$ with overlay structure $\mathcal{S}$, let $G$ be the group of deck transformations of $p$ preserving $\mathcal{S}$. If $X$ is connected, then the action of $G$ on $X$ is an overlay action with overlay structure $\mathcal{S}$. Moreover, $p$ factors as the composition of the projection $X\to X/G$ and an overlay $q:X/G\to Y$ with overaly structure in the form of images of elements of $\mathcal{S}$.
\end{Proposition}
\begin{proof}
Recall $h:X\to X$ is a \textbf{deck transformation} of $p$ if it is a homeomorphism and $p\circ h=p$.
$h$ \textbf{preserves} $\mathcal{S}$ if $h(U)\in \mathcal{S}$ for every $U\in \mathcal{S}$.\\
Suppose $h(x_0)=x_0$ for some $x_0\in X$. Given $x\in X$ pick an $\mathcal{S}$-chain
$\{x_0,\ldots,x_n=x\}$.
The $\mathcal{S}$-chains $\{h(x_0),\ldots,h(x_n)\}$ and $\{x_0,\ldots,x_n\}$ are lifts of the same chain in $Y$, so they are equal. Thus $h(x)=x$. That means $G$ acts freely on $X$.\\
If $h(U_0)\cap U_0\ne\emptyset$ for some slice $U_0\in \mathcal{S}$, then $h(U_0)=U_0$ as both are members of $\mathcal{S}$. Again $h=id_X$ as above. That means the action of $G$ on $X$ is a covering action and each $U\in \mathcal{S}$ is a slice of that action.\\
Suppose $U,V\in \mathcal{S}$. If $q:X\to X/G$ is the projection and $q(U)\cap q(V)=\emptyset$, then clearly $U\cup V$ is a slice of the action of $G$ on $X$. Otherwise, $U\cap h(V)\ne\emptyset$ for some $h\in G$.
We claim $U\cup h(V)$ is a slice of $q$. Indeed, if there is $g\in G\setminus\{1_G\}$ with $(U\cup h(V))\cap g(U\cup h(V))\ne\emptyset$,
then $g(U)$ is disjoint from $U$ resulting in $g(U)\cap h(V)\ne\emptyset$. That means $h(V)$ intersects two different slices $U$ and $g(U)$ over the same set $p(U)$, a contradiction.
\par
The proof that $q:X/G\to Y$ has overaly structure in the form of images of elements of $\mathcal{S}$ is similar to that of the action of $G$ being an overlay action.
\end{proof}

\section{Regular overlays}
\begin{Definition}
An overlay structure $\mathcal{S}$ of $p:X\to Y$ with overlay cover $\mathcal{U}$ is called \textbf{regular}
if there is no $\mathcal{U}$-loop with one lift being an $\mathcal{S}$-loop and another lift being a non-loop.
\end{Definition}

\begin{Example}\label{OverlayActionsAreRegular}
Every overlay structure of an overlay action of $G$ on $X$ is regular.
\end{Example}
\begin{proof}
Given a lift $\{x_0,\ldots,x_n\}$ of a chain in the overlay cover of $X/G$, every other lift of that chain is of the form
$\{g\cdot x_0,\ldots,g\cdot x_n\}$. Therefore if $\{x_0,\ldots,x_n\}$ is a loop, so is $\{g\cdot x_0,\ldots,g\cdot x_n\}$.
\end{proof}

\begin{Theorem}
Suppose $\mathcal{S}$ is an overlay structure of $p:X\to Y$ with overlay cover $\mathcal{U}$. If $G$ is the group  of deck transformations of $p$ preserving $\mathcal{S}$ and $X$ is connected, then the following conditions are equivalent: \\
a. $\mathcal{S}$ is a regular overlay structure,\\
b. $G$ acts transitively on each fiber of $p$,\\
c. There is a homeomorphism $h:Y\to X/G$ such that $h\circ p$ is the projection $X\to X/G$.\\
d. Every overlay structure of $p$ is regular.
\end{Theorem}
\begin{proof}
a)$\implies$b). Pick two elements $x,y\in X$ satisfying $p(x)=p(y)$ and pick an $\mathcal{S}$-chain
$\{x_0,\ldots,x_n\}$ from $x$ to $y$. Given any point $z\in X$ choose
an $\mathcal{S}$-chain $\{z_0,\ldots,z_m\}$ from $x$ to $z$. Define $h(z)$ as the end-point of the
lift $\{p(z_0),\ldots,p(z_m)\}$ starting from $y$.
\par
Notice $h(z)$ does not depend of the choice of $\{z_0,\ldots,z_m\}$. Indeed, given another $\mathcal{S}$-chain $\{w_0,\ldots,w_q\}$ from $x$ to $z$, $\{p(z_0),\ldots,p(z_m),p(w_q),\ldots,p(w_0)\}$
is a loop that lifts to a loop starting from $x$. Hence its lift starting from $y$ is a loop resulting in independence of $h(z)$ on the chain $\{z_0,\ldots,z_m\}$ (the lifts of $p(z_m)$ and $p(w_q)$ must be equal as they belong to the same fiber of $p$). The same argument shows $h$ is injective.\\
That also implies $h$ is continuous and open as the slice containing $z$ gets to be mapped to the slice containing $h(z)$. 
Observe $h(x)=y$ as we can choose $\{x\}$ to be the chain in the definition of $h(x)$.
\par
b)$\implies$c). In that case the fiber of $p$ containing $x$ equals $G\cdot x$ and the action of $G$ on $X$ is an overlay action with the same overlay structure as that of $p$ (see \ref{DeckTransformationFactorization}). Consequently, there is a natural bijection
$h:Y\to X/G$ satisfying $h\circ p$ equal to the projection $X\to X/G$. Clearly, $h$ is a homeomorphism.
\par
c)$\implies$d). Use \ref{OverlayActionsAreRegular}.
\par
a) is a special case of d).
\end{proof}

The following is a converse to \ref{OverlayCoversForHomomorphisms}:
\begin{Theorem}\label{ConverseOverlayCoversForHomomorphisms}
Suppose $p:X\to Y$ is a regular overlay, $X$ is connected, and $Y$ is a topological group. If there is an open neighborhood $U$ of $1_Y$ in $Y$ such that the family $\mathcal{U}=\{U\cdot y\}_{y\in Y}$ is an overlay cover of $Y$, then one can put a structure of a topological group on $X$ making $p$ a continuous homomorphism.
\end{Theorem}
\begin{proof} Pick an overlay structure $\mathcal{S}$ over $\mathcal{U}$.
Pick $x_0\in X$ with $p(x_0)=1_Y$. 
\par Suppose $\{y_1,\ldots,y_n\}$ is a $\mathcal{U}$-chain. That means existence of $z_i\in Y$ such that $y_i, y_{i+1}\in U\cdot z_i$ for all $i < n$.
Notice $\{y_1\cdot y,\ldots,y_n\cdot y\}$ is a $\mathcal{U} $-chain for each $y\in Y$.
Indeed, $y_i, y_{i+1}\in U\cdot z_i$ implies
$y_i\cdot y, y_{i+1}\cdot y\in U\cdot z_i\cdot y$, so $\{y_1\cdot y,\ldots,y_n\cdot y\}$
is a $\mathcal{U}$-chain.
\par
\textbf{Claim 1}: If an $\mathcal{S}$-lift of $\{y_1\cdot y,\ldots,y_n\cdot y\}$
 is a loop (respectively, is not a loop), then there is a neighborhood $N(y)$ of $y$ such that for every $t\in N(y)$ every $\mathcal{S}$-lift of the $\mathcal{U}$-loop $\{y_1\cdot t,\ldots,y_n\cdot t\}$
 is a loop (respectively, is not a loop).\\
\textbf{Proof of Claim 1}: If we choose $t$ so close to $y$ that
$y_i\cdot t\in U\cdot z_i$ for all $i\leq n$ and $y_{i+1}\cdot t\in U\cdot z_i$ for $i < n$, then any $\mathcal{S}$-lift of $\{y_1\cdot y,\ldots,y_n\cdot y\}$ induces a lift of $\{y_1\cdot t,\ldots,y_n\cdot t\}$ via the same slices.
Therefore end-points of a lift of $\{y_1\cdot t,\ldots,y_n\cdot t\}$ are in the same slice, hence are equal (respectively, not equal).
$\Box$

\textbf{Claim 2}: If an $\mathcal{S}$-lift of $\{y_1,\ldots,y_n\}$
 is a loop, then for every $t\in Y$ every $\mathcal{S}$-lift of the $\mathcal{U}$-loop $\{y_1\cdot t,\ldots,y_n\cdot t\}$
 is a loop.\\
\textbf{Proof of Claim 2}: Let $\Sigma$ be the set of $t$ such that every $\mathcal{S}$-lift of the $\mathcal{U}$-loop $\{y_1\cdot t,\ldots,y_n\cdot t\}$
 is a loop. By Claim 2 both $\Sigma$ and its complement are open. Since $1_Y\in\Sigma$, $\Sigma=Y$.
$\Box$

\textbf{Claim 3}: For each $x\in X$ there is a unique homeomorphism $h:X\to X$ preserving $\mathcal{S}$ such that
$h(x_0)=x$ and $p(h(z))=p(z)\cdot p(x)$ for each $z\in X$.\\
\textbf{Proof of Claim 3}: 
Given $z\in X$ pick an $\mathcal{S}$-chain $C(z)$ from $x_0$ to $z$.
Lift $p(C_z)\cdot p(x)$ starting from $x$ to obtain $h(z)$ as the end of the lift. Notice $h(z)$ is independent on the chain $C(z)$ by Claim 2. It is injective as $p(C_z)\cdot p(x)$ having the same end-point of the lift as $p(C_t)\cdot p(x)$ implies (Claim 2) that $p(C_z)$ has the same end-point of its lift $C_z$ (it is $z$) as the lift $C_t$ of $p(C_t)$ (it is $t$) - in both cases we lift starting from $x_0$. It is clear $h$ preserves $\mathcal{S}$ and is surjective. Therefore both $h$ and $h^{-1}$ are continuous.
$\Box$

\textbf{Proof of \ref{ConverseOverlayCoversForHomomorphisms}}:
Let $G$ be the group of homeomorphisms $h$ of $X$  preserving $\mathcal{S}$ such that there is $y_h\in Y$ with $(p\circ h)(x)= p(x)\cdot y_h$
for each $x\in X$. Let $\phi:G\to X$ be the evaluation function $\phi(h)=h(x_0)$. By Claim 3, $\phi$ is surjective. Observe $h(x_0)=g(x_0)$ implies $y_h=y_g$ and $h=g$.
Thus $\phi$ is bijective and we can use it to give $X$ the desired structure of a topological group.
\end{proof}

\section{Overlays versus coverings}

\begin{Theorem}\label{MainThmOverlays}
Suppose $\mathcal{S} $ is a covering structure of $p:X\to Y$ and $\mathcal{U} =p(\mathcal{S} )$. The following conditions are equivalent:\\
a. $\mathcal{S} $ is an overlay structure of $p:X\to Y$,\\
b. the induced map $\mathcal{N} (p):\mathcal{N} (\mathcal{S} )\to \mathcal{N} (\mathcal{U} )$ of nerves of covers is a covering map.
c. the induced map $\mathcal{N} (p):\mathcal{N} (\mathcal{S} )^{(1)}\to \mathcal{N} (\mathcal{U} )^{(1)}$ of $1$-skeleta of nerves of covers is a covering map.
\end{Theorem}
\begin{proof}
\textbf{a)$\implies$b)}. Pick $U\in \mathcal{U} $ and consider the open star $st(U,\mathcal{N} (\mathcal{U} ))$. We will show it is evenly covered by $\mathcal{N} (p):\mathcal{N} (\mathcal{S} )\to \mathcal{N} (\mathcal{U} )$. Pick $V\in \mathcal{S} $ so that $p(V)=U$. Given any $U_i\in \mathcal{U} $ intersecting $U$ there is exactly one $V_i\in \mathcal{S} $ intersecting $V$ such that $p(V_i)=U_i$. Also, if $y\in U\cap \bigcap\limits_{i=1}^n U_i$, then 
there is unique $x\in V$ satisfying $p(x)=y$. Notice
 $x\in V_i\cap p^{-1}(y)$ as $st(x,\mathcal{S} )$ is a slice of $p$.
That means $p^{-1}(st(U,\mathcal{N} (\mathcal{U} )))$ is the disjoint union of $st(V,\mathcal{N} (\mathcal{S} ))$, $V$ ranging over all elements of $\mathcal{S} $ whose image is $U$, and each of $st(V,\mathcal{N} (\mathcal{S} ))$ is mapped homeomorphically onto $st(U,\mathcal{N} (\mathcal{U} ))$.
\par
\textbf{b)$\implies$c)} is obvious.
\par
\textbf{c)$\implies$a)}. Suppose $V_i\in \mathcal{S}$, $1\leq i\leq 2$, both intersect $V\in \mathcal{S} $ and $p(V_1)=p(V_2)$. That means the edges $[V,V_1]$ and $[V,V_2]$ in $\mathcal{N} (\mathcal{S} )^{(1)}$ are mapped to the edge
$[p(V),p(V_1)]$ in $\mathcal{N} (\mathcal{U} )^{(1)}$ resulting in $V_1=V_2$.
\end{proof}

\begin{Corollary}\label{OverlaysArePullBacks}
Suppose $p:X\to Y$ is a map of connected spaces.
If $Y$ is paracompact, then $p:X\to Y$ is an overlay if and only if it is the pull-back of a simplicial covering.
\end{Corollary}
\begin{proof}
Notice every simplicial covering is an overlay and the pull-back of an overlay is an overlay.
\par
Suppose $p:X\to Y$ is an overlay with overlay structure $\mathcal{S} $ and the corresponding overlay cover $\mathcal{U} $. Pick a point-finite partition of unity $\phi=\{\phi_U\}_{U\in \mathcal{U} }$ on $Y$ such that $\phi_U(Y\setminus U)\subset \{0\}$
for each $U\in \mathcal{U} $. By \ref{MainThmOverlays} the natural map $\mathcal{N} (p):\mathcal{N} (\mathcal{S} )\to \mathcal{N} (\mathcal{U} )$ is a simplicial covering. We claim $p$ is the pull-back of $\mathcal{N} (p)$
under $\phi:Y\to \mathcal{N} (\mathcal{U} )$.
\par
Given $V\in \mathcal{S} $, define $\psi_V:X\to [0,1]$ via $\psi_V(x)=\phi_U(p(x))$, where $U=p(V)$.
Notice $\psi=\{\psi_V\}_{V\in \mathcal{S} }:X\to \mathcal{N} (\mathcal{S} )$ is a partition of unity
on $X$ such that $N(p)\circ \psi=\phi\circ p$.
\par
Suppose $y\in Y$ and $z=\sum_{V\in S} c_V\cdot V\in \mathcal{N} (\mathcal{S})$
such that $\phi(y)=\mathcal{N} (p)(z)$. That means $c_V=\phi_{p(V)}(y)$ for each $V\in \mathcal{S} $.
Indeed, one cannot have $p(V)=p(W)$ and $c_V\ne 0$, $c_W\ne 0$ for two different elements $V$ and $W$
of $\mathcal{S} $. If $c_V\ne 0$, pick unique $x_V\in V$ such that $p(x_V)=y$. As the intersection of all
$V\in S$ satisfying $c_V\ne 0$ contains some element $w$ of $X$ and $st(w,\mathcal{S} )$ is a slice, we conclude
$x_V=x_W$ for all $V,W\in \mathcal{S} $ satisfying $c_V\ne 0$, $c_W\ne 0$. That means $x=x_V$
is the unique point of $X$ for which $p(x)=y$ and $\psi(x)=z$. That proves $p$ is the pull-back of
$\mathcal{N} (p)$ under $\phi$ as $p$ is an open map.
\end{proof}

\begin{Remark}
In case of overlays over compact metric spaces, \ref{OverlaysArePullBacks} was proved in \cite{Mel}  (Lemma 7.3).
\end{Remark}

\begin{Corollary}\label{StructureOfGroupOnOverlay}
If $p:X\to Y$ is an overlay, $X$ is connected, and $Y$ is a compact topological group, then one can put a structure of a topological group on $X$ making $p$ a continuous homomorphism and $ker(p)$ finitely generated and Abelian.
\end{Corollary}
\begin{proof}
 $p$ is the pull-back under a map $f:Y\to K$ of a simplicial covering.
Now, $f$ can be factored (up to homotopy) through a compact Lie group $K'$ and a continuous homomorphism $h:Y\to K'$ as $Y$ is the inverse limit of compact Lie groups.
The corresponding connected covering $q:L'\to K'$ is regular (the fundamental group of $K'$ is Abelian)
and 
the space $L'$ can be converted to a topological group (use \ref{ConverseOverlayCoversForHomomorphisms}) with the kernel being finitely generated and Abelian (as it is isomorphic to a subgroup of a finitely generated Abelian group $\pi_1(K')$). The pull-back of $q$ under $h$ gives rise
to a structure of a topological group on $X$ with the same kernel as that of $q$.
\end{proof}

\begin{Remark}
Eda-Matijevi\' c \cite{EM2}  proved \ref{StructureOfGroupOnOverlay} without concluding $ker(p)$ is finitely generated and Abelian. In case of $Y$ being a solenoid they proved $ker(p)$ is finite.
\end{Remark}

\end{document}